\documentclass[12pt]{amsart}
\usepackage{amssymb,amsmath,amsthm,mathtools}
\usepackage{amsthm,amssymb,amstext,amscd,amsfonts,amsbsy,amsxtra,latexsym,amsmath,xcolor,mathrsfs,fancybox}
\usepackage{fullpage}
\usepackage[english]{babel}
\usepackage[latin1]{inputenc}
\usepackage{cancel}
\usepackage[draft]{hyperref}
\usepackage{comment,enumitem}
\usepackage{mdframed}
\usepackage[top=1.5in, bottom=2in, left=2.5in, right=2.5in]{geometry}
\allowdisplaybreaks

\numberwithin{equation}{section}
\newtheorem{theorem}{Theorem}
\numberwithin{theorem}{section}

\newtheorem{corollary}[theorem]{Corollary}
\newtheorem{proposition}[theorem]{Proposition}

\makeatletter
\newcommand{\vast}{\bBigg@{3}}
\makeatother

\theoremstyle{remark}

\newtheorem*{remark}{Remark}

\newtheorem*{definition}{Definition}
\numberwithin{theorem}{section}
\numberwithin{equation}{section}

\newcommand{\N}{\mathbb{N}}
\newcommand{\Z}{\mathbb{Z}}
\newcommand{\R}{\mathbb{R}}
\newcommand{\C}{\mathbb{C}}

\newcommand{\Q}{\mathbb{Q}}

\newcommand{\SL}{{\text {\rm SL}}}

\newcommand{\sgn}{\operatorname{sgn}}

\def\H{\mathbb{H}}
\newcommand{\sm}[4]{\left(\begin{smallmatrix}#1&#2\\ #3&#4
	\end{smallmatrix} \right)}

\begin{document}
\title{Taylor coefficients of non-holomorphic Jacobi forms and applications}

\thanks{The research of the author is supported by the Alfried Krupp Prize for Young University Teachers of the Krupp foundation and the research leading to these results receives funding from the European Research Council under the European Union's Seventh Framework Programme (FP/2007-2013) / ERC Grant agreement n. 335220 - AQSER}


\author{Kathrin Bringmann}

\begin{abstract}
In this paper, we prove modularity results of Taylor coefficients of certain non-holomorphic Jacobi forms. It is well-known that Taylor coefficients of holomorphic Jacobi forms are quasimoular forms. However recently there has been a wide interest for Taylor coefficients of non-holomorphic Jacobi forms for example arising in combinatorics. In this paper, we show that such coefficients still inherit modular properties. We then work out the precise spaces in which these coefficients lie for two examples.
\end{abstract}

\maketitle
\begin{center}
{\it In honor of Don Zagier who has been a great inspiration.}
\end{center}

\section{Introduction and statement of results}

As motivating example, we start with the
generating function for partitions. As usual let $p(n)$ denote the number of partitions of $n$. By Euler, we have
\begin{eqnarray}\label{eq:partgen}
P(q):= \sum_{n=0}^{\infty} p(n) q^{n}
= \frac{q^{\frac{1}{24}}}{\eta(\tau)},
\end{eqnarray}
where
$\eta(\tau):=
q^{ \frac{1}{24}}\, \prod_{n=1}^{\infty}(1-q^n)$
is {\it Dedekind's $\eta$-function}, a weight $\frac12$ modular form
($q:=e^{2 \pi i \tau}$ throughout).
Of the the many consequences of the modularity properties of $P,$ one of the most striking ones is
the three congruences due to Ramanujan,
namely
\begin{align*}
p(5n+4) & \equiv 0 \pmod 5,\\
p(7n+5) & \equiv 0 \pmod 7, \\
p(11n+6) & \equiv 0 \pmod {11}.
\end{align*}

To explain the congruences with modulus $5$ and $7$, Dyson \cite{Dy}
introduced the \textit{rank} of a partition, which is defined to be
its largest part minus the number of its parts.
Dyson conjectured that the partitions of $5n+4$ (resp. $7n+5$) form
$5$ (resp. $7$) groups of equal size if sorted by their ranks
modulo $5$ (resp. $7$).
This conjecture was proven by Atkin and Swinnerton-Dyer \cite{AS}.
If $N(m,n)$ denotes the number of partitions of $n$  with rank $m$,
then
we have the generating  function
\begin{equation}\label{rank}
\begin{aligned}
R(\zeta;q):=\,& 1 + \sum_{m \in \Z}
\sum_{n=1}^{\infty}
N(m,n) \zeta^m q^n
\\=\,&1 + \sum_{n=1}^\infty \frac{q^{n^2}}{(\zeta q;q)_n (\zeta^{-1}q;q)_n}
= \frac{(1-\zeta)}{(q;q)_{\infty}} \sum_{n \in \Z}
\frac{(-1)^n q^{\frac{n}{2}(3n+1) }}{1-\zeta q^n},
\end{aligned}
\end{equation}
where $(a;q)_{n}:= \prod_{j=0}^{n-1}(1-aq^j)$ and $(a;q)_{\infty}:=
\lim_{n \to \infty} (a;q)_n $.
In particular
\begin{align*}
R(1;q)&=P(q), \\
R(-1;q)&=f(q):=1+\sum_{n=1}^{\infty}\frac{q^{n^2}}{(-q;q)_n^2 }.
\end{align*}
The function $f(q)$ is one of the \textit{mock theta functions},
defined by Ramanujan in his last letter to Hardy. Such a mock theta function can be completed by adding a non-holomorphic piece to obtain a harmonic Maass form. These non-holomorphic automorphic forms generalize classical modular forms. Instead of being holomorphic they are annihilated by a hyperbolic Laplace operator; see Subsection 2.3 for the precise definition. More generally, $R(\zeta;q)$ is a (non-holomorphic) Jacobi form.

The Taylor coefficients of $R(\zeta;q)$ are of combinatorial interest.
To be more precise, Atkin and Garvan \cite{AG} introduced the
{\it $k$-th rank moments}
\begin{equation*}
N_k(n) := \sum_{m \in \Z} m^k N(m,n).
\end{equation*}
Andrews \cite{An2} then studied
  the    \textit{symmetrized $k$-th rank moment function}
\begin{equation*}
\eta_k(n):= \sum_{m  \in\Z}
\left(
\begin{matrix}
m + \Big[\frac{k-1}{2} \Big] \\
k
\end{matrix}
\right)
N(m,n).
\end{equation*}
Using the symmetry $N(-m,n) =
N(m,n)$, one sees that one only needs to consider even rank moments. For these we define the rank
generating function
$$
\mathcal{N}_{2\ell}(q):= \sum_{n=0}^\infty\left(\sum_{m\in \Z}m^{2\ell}N(m,n)\right)q^n.
$$
Summing up these moments gives the whole rank generating function $(\zeta:=e^{2\pi iz})$
\begin{equation}\label{TaylorR}
R(\zeta;q) =\sum_{\ell =0}^\infty\mathcal{N}_{2\ell}(q) \frac{\left(2\pi i z\right)^{2\ell}}{(2\ell)!}.
\end{equation}
So the $\mathcal N_{2\ell}$ are basically the Taylor coefficients of $R(\zeta;q)$.

Let us next describe what is known about modularity of these functions.  Firstly
$$
\mathcal{N}_0(q)=P(q)=\frac{1}{(q)_\infty},
$$
so this case yields (up to a $q$-power) a weight $-\frac12$ modular form.
For the next example, $\mathcal N_2$, it was shown in \cite{Br} that it can be "completed" to a harmonic Maass form of weight $\frac32$. Moreover the general $R_k$ are  known to give linear combinations of derivatives of harmonic Maass forms, see \cite{BGM}.
However, in particular for applications one requires the pure modular objects as well as their behavior under differential operators.
In this paper, we explicitly give the completed modular object as well as the space in which the function lie.

For this define
$$
r_{2\ell-1}= r^+_{2\ell-1}+r^-_{2\ell-1}
$$
with
\begin{align}\label{splus}
r^+_{2\ell-1}(\tau)&:=(2\pi i)^{2\ell-1}\sum_{0\leq j + n\leq \ell}\frac{B_{2n}\left(\frac12\right)}{(2n)!(\ell-j-n)!}\left(\frac{E_2(\tau)}{8}\right)^{\ell-j-n}\frac{\mathcal{N}_{2j}(q)q^{-\frac{1}{24}}}{(2j)!}, \\
\label{sminus}
r^-_{2\ell-1}(\tau) &:= \frac{1}{(2\ell-1)!} \left[\frac{\partial^{2\ell-1}}{\partial z^{2\ell-1}}\left(\zeta^{-1} q^{-\frac16} S(3z+\tau; 3\tau)e^{-\frac{\pi^2}{2}E_2(\tau)z^2}\right)\right]_{z=0}.
\end{align}
Here $B_n(x)$ is the $n$-th Bernoulli polynomial, $E_2$ is the weight $2$ (quasimodular) Eisenstein series given in \eqref{E2holdef} and we also require the even function ($\tau = u+iv, u,v\in \R$)
\begin{equation}\label{S}
S(z; \tau):=\sum_{n\in\frac12+\Z}\left(\sgn(n)-E\left(\left(n+\frac{y}{v}\right)\sqrt{2v}\right)\right)
(-1)^{n-\frac12} q^{-\frac{n^2}{2}} e^{-2\pi inz},
\end{equation}
where
$
E(z):=2\int_0^z e^{-\pi t^2}dt.
$

\begin{theorem}\label{rankmain}
For $\ell\in \N$, $r_{2\ell-1}$ transforms for all $(\begin{smallmatrix} a&b\\c&d\end{smallmatrix})\in \SL_2(\Z)$ as
\begin{equation}\label{ST}
r_{2\ell-1} \left(\frac{a\tau+b}{c\tau+d}\right)=\psi^{-1} \left(\begin{matrix}a&b\\c&d\end{matrix}\right) (c\tau+d)^{2\ell-\frac12} r_{2\ell-1}(\tau),
\end{equation}
where $\psi$ is the multiplier of $\eta$ (see \eqref{multeta}). Moreover
\begin{equation}\label{Lsl}
	L(r_{2\ell-1}(\tau)) =\frac{i\sqrt{3}}{\sqrt{2}}\frac{v^{\frac12}\overline{\eta(-\tau)}}{(\ell-1)!}\left(-\frac{\pi^2\widehat{E_2}(\tau)}{2}\right)^{\ell-1},
\end{equation}
where $L:= -2 iv^2 \frac{\partial}{\partial \overline{\tau}}$ is the {Maass lowering operator}.
\end{theorem}
\begin{remark}
Instead of $E_2$ in \eqref{splus} and \eqref{sminus} one can also use $\frac1v$. The completed object $\widetilde{r}_{2\ell-1}$ obtained that way basically maps to $\widetilde{r}_{2\ell-3}$ under lowering.

\end{remark}
Theorem \ref{rankmain} easily implies the following main result from \cite{Br}. To state this set
\[
\mathcal{M}(\tau):=\mathcal{R}(\tau)+\mathcal{N}(\tau)-\frac{1}{24\eta(24\tau)}+\frac{E_2(24\tau)}{8\eta(24\tau)}
\]
with
\[
\mathcal{R}(\tau):=\frac12R_2\left(q^{24}\right)q^{-1},\quad\mathcal{N}(\tau):=\frac{i}{4\sqrt{2}\pi}\int_{-\overline{\tau}}^{i\infty}\frac{\eta(24w)}{(-i(\tau+w))^{\frac{3}{2}}} dw.
\]
\begin{corollary}\label{DukeCor}
The function $\widehat M(\tau):= \mathcal M (\frac{\tau}{24})$ is a harmonic Maass form of weight $\frac32$ transforming, for all $(\begin{smallmatrix} a&b\\c&d \end{smallmatrix}) \in \SL_2(\Z)$, as
$$
\widehat M\left( \frac{a\tau+b}{c\tau+d}\right)=\psi^{-1} \left(\begin{matrix}a&b\\c&d\end{matrix}\right) (c\tau+d)^\frac32  \widehat M(\tau).
$$
\end{corollary}

Our second example comes from Joyce invariants. In \cite{MO}, the authors studied the (slightly modified) generating function of Joyce invariants, which turn out to be the following $q$-series ($k\in2\N$)
$$
\mathbb{J}_k(\tau):=\frac12\sum_{n\in\Z\setminus\{0\}}\frac{n^{k-1}q^{n^2}}{1-q^n}.
$$
Define
\begin{equation*}
\widehat{\mathbb{J}}_k:=\mathbb{J}_k+\frac{\delta_{k=2}}{8\pi v}+\frac{(k-2)!(-1)^{\frac{k}{2}+1}}{\Gamma\left(\frac{k-1}{2}\right)^2 2^{k+1}}\sum_{\nu\in\{-1,0\}}\left[\vartheta_\nu, s_\nu\right]_{\frac{k}{2}-1}
\end{equation*}
with Rankin Cohen brackets $[.,.]_\kappa$ as in \eqref{RC} and
\begin{equation}\label{frn}
\begin{aligned}
\vartheta_\nu(\tau)&:=q^{\frac{\nu^2}{4}}\vartheta\left(\nu\tau+\frac12; 2\tau\right),\qquad \\
s_\nu(\tau)&:=\sqrt{\pi}\sum_{n\in\frac{\nu+1}{2}+\Z} \left\lvert
n+\frac{\nu}{2}\right\rvert \Gamma\left(-\frac12, 4\pi\left(n+\frac{\nu}{2}\right)^2 v\right)q^{-\left(n+\frac{\nu}{2}\right)^2}.
\end{aligned}
\end{equation}
Here $\vartheta$ is the Jacobi theta function given in \eqref{thedefn} and $\Gamma(\alpha,x):= \int_x^\infty e^{-t} t^{\alpha-1} dt$ is the incomplete Gamma function.
\begin{remark}
	We have that
	\begin{equation}\label{rewritetheta}
	\vartheta_{-1}(\tau)=-\theta_1(2\tau), \quad \vartheta_0(\tau)=-\theta_3(2\tau),
	\end{equation}	
	where
	$$
	\theta_1(\tau):=\sum_{n\in\Z}q^{\frac{n^2}{2}}, \quad \theta_3(\tau):=\sum_{n\in\Z}q^{\frac{1}{2}\left(n+\frac12\right)^2}.
	$$
\end{remark}
\begin{theorem}\label{joycecomple}
For all $M=(\begin{smallmatrix} a&b\\c&d\end{smallmatrix})\in \SL_2(\Z)$, the following transformation law holds
\begin{equation*}
\widehat{\mathbb J}_k\left(\frac{a\tau+b}{c\tau+d}\right) =(c\tau+d)^k\widehat{\mathbb J}_k(\tau).
\end{equation*}
We have
\begin{equation}\label{lower}
L\left(\widehat{\mathbb J}_k(\tau)\right)=-\frac{\delta_{k=2}}{8\pi}-\frac{i(k-1)}{8(2\pi i)^{k-1}} \sqrt{v}
\left(\overline{\theta_1(2\tau)}\vartheta_{k-1, -1}(\tau)+\overline{\theta_3(2\tau)}\vartheta_{k-1, 0}(\tau)\right),
\end{equation}
where
\[
\vartheta_{\ell, \nu}(\tau):=\left[\frac{\partial^{\ell-1}}{\partial z^{\ell-1}}\left(\vartheta\left(z+\nu\tau+\frac12; 2\tau\right)e^{\pi i\nu z} q^{\frac{\nu^2}{4}} e^{\frac{\pi z^2}{4v}}\right)\right]_{z=0}
\]
is an almost holomorphic modular form of weight $\ell-\frac12$ (with some multiplier).
\end{theorem}
In particular, setting $k=2$, we recover the result by Mellit and Okada, slightly rewritten.

\begin{corollary}\label{Mellit}
We have that $\widehat{\mathbb J}_2$ transforms of weight $2$ and
\[
L\left(\widehat{\mathbb J}_2(\tau)\right)=-\frac1{4\pi}+\frac1{16\pi v}
\left(\overline{\theta_1(2\tau)}\theta_1(2\tau)+\overline{\theta_3(2\tau)}\theta_3(2\tau)\right).
\]
\end{corollary}

The paper is organized as follows. In Section \ref{sec:prelim} we recall basis facts about Rankin-Cohen brackets, Jacobi forms, harmonic Maass forms, and Appell functions. In Section \ref{IdeaOfApproach}, we describe modularity properties of (general) non-holomorphic Jacobi forms. We then carry on this approach in Section \ref{4} for rank moments and in Section \ref{sec:joyce} for Joyce invariants. For the readers convinience, we keep Sections \ref{4} and \ref{sec:joyce} independent of Section \ref{IdeaOfApproach}.

\section*{Acknowledgements}
The author thanks Karl Mahlburg for many enlightening discussions.
\section{Preliminaries}\label{sec:prelim}
\subsection{Rankin-Cohen brackets}\label{2.1}
For $f_1,f_2$  modular forms of weight $k_1,k_2\in\frac12\Z$, respectively, define for $\nu\in\N_0$ the \textit{$\nu$th Rankin-Cohen bracket}
\begin{equation}\label{RC}
[f_1,f_2]_\nu:=\sum_{j=0}^\nu (-1)^j \binom{k_1+\nu-1}{\nu-j} \binom{k_2+\nu-1}{j} D^j(f_1)D^{\nu-j}(f_2)
\end{equation}
with $D:=\frac1{2\pi i}\frac{\partial}{\partial \tau}$. We then have that $[f_1,f_2]_\nu$ is a modular form of weight $k_1+k_2+2\nu$. In this paper we also allow $f_2$ to be the non-holomorphic part of a harmonic Maass form (see Subsection \ref{2.3} for the definition).
\subsection{Classical Jacobi forms and quasimodular forms}\label{2.2}
We first recall the definition of holomorphic Jacobi forms, following Eichler and Zagier \cite{EZ}.
\begin{definition} A {\it holomorphic Jacobi form of weight $k$ and index $m$} ($k, m \in \mathbb N$) on a subgroup $\Gamma \subseteq \textnormal{SL}_2(\mathbb Z)$ of finite index is a holomorphic function
	$\varphi:\mathbb C \times \mathbb H \to
	\mathbb C$ which, for all $\gamma = \sm{a}{b}{c}{d} \in \Gamma$ and $\lambda,\mu \in \mathbb Z$, satisfies
	\begin{enumerate}\item
		$\varphi\left(\frac{z}{c\tau + d};\frac{a\tau+b}{c\tau+d}\right) = (c\tau + d)^k e^{\frac{2\pi i m c z^2}{c\tau + d}} \varphi(z;\tau)$,
		\item $\varphi(z + \lambda \tau + \mu;\tau) = e^{-2\pi i m (\lambda^2\tau + 2\lambda z)} \varphi(z;\tau)$,
		\item $\varphi$ has a Fourier development of the form $\sum_{n, r} c(n,r)q^n e^{2\pi i r z}$ with \\
		$c(n,r)=0$ unless $n\geq r^2/4m$.
	\end{enumerate}
	\end{definition}
\noindent Denote by $J_{k,m}$ the space of holomorphic Jacobi forms of weight $k$ and index $m$.
	Jacobi forms with multipliers and of half-integral weight are defined similarly with obvious modifications made.
	
	A special Jacobi form used in this paper is
	\begin{equation}
	\vartheta(z)= \vartheta(z;\tau) :=\sum_{n\in\frac12+\Z}e^{\pi i n^2\tau+2\pi in\left(z+\frac12\right)}, \label{thedefn}
	\end{equation}
	where  here and throughout, we may omit the dependence of various functions on the variable $\tau$ if
	the context is clear.
	This function  is well-known to satisfy the following transformation law \cite[(80.31) and (80.8)]{Rad}.
	\begin{proposition}\label{THETtrans}
		For $\lambda,\mu \in \mathbb Z$ and
		$\gamma=\sm{a}{b}{c}{d} \in \textnormal{SL}_2(\mathbb Z)$, we have that
		\begin{align}
		\vartheta(z+\lambda \tau+\mu)&=(-1)^{\lambda+\mu}q^{-\frac{\lambda^2}{2}}e^{-2\pi i\lambda
			z}\vartheta(z),\label{tt1}\\ \vartheta\left(\frac{z}{c\tau+d};
		\frac{a\tau+b}{c\tau+d}\right)&=\psi^3\left(\gamma\right) (c\tau+d)^{\frac12}e^{\frac{\pi icz^2}{c \tau+d}}\vartheta(z;
		\tau)\label{tt2},
		\end{align}
		with $\psi$ is the multiplier of Dedekind's $\eta$-function, i.e.,
		\begin{equation}\label{multeta}
		\eta\left(\frac{a\tau+b}{c\tau+d}\right)=\psi(\gamma)(c\tau+d)^\frac12\eta(\tau).
		\end{equation}
		
	\end{proposition}
	
	The Jacobi theta function is also known to satisfy the well known triple product identity
	\begin{align*}
	\vartheta(z;\tau) = -i q^{\frac{1}{8}} \zeta^{-\frac{1}{2}} \prod_{n=1}^\infty (1-q^n)\left(1-\zeta q^{n-1}\right) \left(1-\zeta^{-1}q^n\right).
	\end{align*}	

The Taylor coefficients in $z$ of Jacobi forms are quasimodular forms which we next recall.
These are holomorphic parts of \emph{almost holmorphic modular forms}, which as originally defined by Kaneko-Zagier \cite{KZ},  transform like usual modular forms,
but are polynomials in $1/v$ with holomorphic coefficients. The holomorphic parts are in particular the constant terms of these polynomials. In this paper we use a slightly modified definition allowing weakly holomorphic coefficients. We let $\widehat{M}_k(\Gamma,\chi)$ denote the space of almost holomorphic modular forms of weight $k$ for $\Gamma\subset\SL_2(\Z)$ and multiplier $\chi$.
Standard examples of almost holomorphic modular
form include derivatives of holomorphic modular forms (corrected so it transforms modular), as well as the non-holomorphic Eisenstein series $\widehat{E}_2$, defined by
\begin{align*}
\widehat{E}_2(\tau) := E_2(\tau) - \frac{3}{\pi v}.
\end{align*}
Here its holomorphic part is given by
\begin{align}\label{E2holdef}E_2(\tau) := 1-24\sum_{n= 1}^\infty \sigma_1(n) q^n,
\end{align}
where $\sigma_1(n)$ is the sum of positive integer divisors of $n$.
The function $E_2$ satisfies, for $\left(\begin{smallmatrix}
a & b \\ c & d
\end{smallmatrix}\right)\in\SL_2(\Z)$,
\begin{equation}\label{E2trans}
E_2\left(\frac{a\tau+b}{c\tau+d}\right)=(c\tau+d)^2E_2(\tau)-\frac{6ic}{\pi}(c\tau+d).
\end{equation}

As mentioned above, Taylor coefficients of Jacobi forms are quasimodular forms. To be more precisely, for
$\phi \in J_{k,m}$ we write its Taylor expansion as
$$
\phi(z;\tau) =: \sum_{n = 0}^\infty \chi_n(\tau) z^n.
$$
\begin{proposition} \label{almostcomplete}
The function
	$$
	\psi_n(\tau):= \sum_{0 \leq j \leq \frac{n}{2}}\frac{\left(\frac{\pi m}{v}\right)^j}{j!}\chi_{n-2j}(\tau)
	$$
	is an almost holomorphic modular form of weight $k+n$.
\end{proposition}

Alternatively one can find a ``modular completion'' using $E_2$.
\begin{proposition}
	We have that
	$$
	\rho_n(\tau):= \sum_{0 \leq j \leq \frac{n}{2}}\frac{\left(\frac{\pi^2 m}{3} E_2(\tau)\right)^j}{j!}\chi_{n-2j}(\tau)
	$$
	is a modular form of weight $k+n$.
\end{proposition}
\begin{remark}
	There are further ways of completing Taylor coefficients of Jacobi forms to modularity objects, namely using derivatives of previous coefficients or Rankin-Cohen brackets.
\end{remark}
\subsection{Harmonic Maass forms}\label{2.3}
We next recall non-holomorphic generalizations of modular forms, following Bruinier-Funke \cite{BF}.
\begin{definition}
For $k\in\frac12\Z$, a {\it weight $k$ harmonic Maass form} on a congruence subgroup $\Gamma\subset \mathrm{SL}_2(\Z)$ is any smooth function $f:\H\to\C$ satisfying the following properties:\\
\noindent
(1)\
For all $\left(\begin{smallmatrix} a&b\\c&d\end{smallmatrix}\right)\in\Gamma$
\[
f\left(\frac{a\tau+b}{c\tau+d}\right)=
\begin{cases}
(c\tau+d)^k f(\tau)&\quad\text{ if } k\in\Z,\\
\left(\frac{c}{d}\right) \varepsilon_d^{-2k}(c\tau+d)^kf(\tau)&\quad\text{ if } k\in\frac12+\Z.
\end{cases}
\]
Here $(\frac{c}{d})$ is the extended Legendre symbol and $\varepsilon_d$ is defined for odd integers $d$ as
\[
\varepsilon_d:=
\begin{cases}
1&\quad\text{ if } d\equiv 1\pmod{4},\\
i&\quad\text{ if } d\equiv 3\pmod{4}.
\end{cases}
\]
\noindent
(2)\
We have
\[
\Delta_k(f)=0,
\]
with the {\it weight $k$ hyperbolic Laplace operator}
\[
\Delta_k:=-v^2\left(\frac{\partial^2}{\partial u^2}+\frac{\partial^2}{\partial v^2}\right)+ikv\left(\frac{\partial}{\partial u}+i\frac{\partial}{\partial v}\right).
\]
\noindent
(3)\
There exists a polynomial $P_f(\tau)\in\C[q^{-1}]$ such that
\[
f(\tau)-P_f(\tau)=O\left(e^{-\varepsilon v}\right)
\]
as $v\to\infty$ for some $\varepsilon>0$. Analoguous conditions hold at all cusps of $\Gamma$.

We let $H_k(\Gamma)$ denote the space of harmonic Maass forms of weight $k$ for $\Gamma$. Again one can modify the above definition to include multipliers.

Harmonic Maass forms are in many ways related to classical (weakly holomorphic) modular forms, i.e., those meromorphic modular forms whose only poles may lie at the cusps of $\Gamma$. To describe the connections, define
\[
\xi_k:=2iv^k\frac{\overline{\partial}}{\partial\overline{\tau}}.
\]
Note that
\[
\xi_k=v^{k-2}\overline{L}.
\]
We have, with $S_\kappa(\Gamma) $ denoting the space of weight $\kappa$ for $\Gamma$
\[
\xi_k: H_k(\Gamma)\to S_{2-k}(\Gamma).
\]
If $k\leq\frac32$, we have for $f\in H_k$
$$
f(\tau)=\sum_{n\in \Q\atop{n\gg -\infty}}c_f^+(n)q^n+\sum_{n\in \Q\atop{n<0}}c_f^-(n)\Gamma(1-k,-4\pi nv)q^n.
$$
\end{definition}
\subsection{Appell function}
Let
\begin{equation}\label{defineAl}
A_\ell\left(z_1,z_2\right)=A_\ell\left(z_1, z_2; \tau\right):=e^{\pi i\ell z_1}\sum_{n\in\Z}
\frac{(-1)^{\ell n}e^{2\pi inz_2} q^{\frac{\ell n(n+1)}{2}}}{1-e^{2\pi iz_1}q^n}.
\end{equation}
We next recall the completion $\widehat{A}_\ell$ from \cite{Zw}, namely
\begin{multline*}
\widehat{A}_\ell\left(z_1, z_2; \tau\right):= A_\ell\left(z_1, z_2; \tau\right)\\
+\frac{i}{2}\sum_{\nu=0}^{\ell-1} e^{2\pi i\nu z_1}
\vartheta\left(z_2+\nu\tau+\frac{\ell-1}{2}; \ell\tau\right)S\left(\ell z_1-z_2-\nu\tau-\frac{\ell-1}{2}; \ell\tau\right)
\end{multline*}
with $S$ and $\vartheta$ defined in \eqref{S} and \eqref{thedefn}, respectively.

The function $\widehat{A}_\ell$ transforms as a multivariable Jacobi form.

\begin{theorem}
\label{2.4}
We have, for $n_1,n_2,m_1,m_2\in\Z$
\begin{multline}\label{Appellt}
\widehat{A}_{\ell} \left( z_1 +n_1\tau+m_1,  z_2 +n_2\tau+m_2 \right) \\=(-1)^{\ell\left( n_1 + m_1\right)} e^{2\pi i z_1 \left( \ell n_1 -n_2 \right)} e^{-2\pi i n_1 z_2} q^{\frac{\ell n_1^2}2 - n_1 n_2} \widehat{A}_\ell \left(z_1, z_2 \right).
\end{multline}
Moreover, for $\left(\begin{smallmatrix}a&b\\c&d\end{smallmatrix}\right)\in\SL_2(\Z)$,
\begin{equation}\label{Amod}
\widehat{A}_\ell \left( \frac{z_1}{c \tau +d}, \frac{z_2}{c\tau +d}; \frac{a\tau+b}{c\tau+d} \right) = (c\tau+d) e^{\frac{\pi i c\left( -\ell z_1^2 +2z_1z_2\right)}{c\tau+d}} \widehat{A}_\ell \left(z_1, z_2; \tau\right).
\end{equation}
\end{theorem}
\begin{remark}
	Plugging in torsion points $z_j\in\Z+\Z\tau$ yields linear combinations of harmonic Maass forms multiplied by modular forms.
\end{remark}

\section{Idea of approach}
\label{IdeaOfApproach}
We consider non-holomorphic functions $\phi : \C\times\H\rightarrow\C$ satisfying the same transformation as Jacobi forms and which have expansions of the form
$$
\phi(z;\tau) =: \sum_{n}\chi_n(\overline{z};\tau)z^n.
$$
Note that we also allow $\chi_n$ to be non-holomorphic in $\tau$. We now describe how to generalize the two approaches of quasimodularity of Taylor coefficients of holomorphic Jacobi forms described in Section \ref{2.2}. For simplicity we restrict to the full modular group and no multiplier.

We first determine non-holomorphic linear combinations of the Taylor coefficients which transform like modular forms.
\begin{proposition}\label{nonholP}
	We have that
	$$
	\psi_n(\tau) := \sum_{j \geq 0} \frac{\left(\frac{\pi m }{v}\right)^j}{j!}\chi_{n-2j}(0;\tau)
	$$
	transforms like a modular form of weight $k+n$.
\end{proposition}
\begin{proof}
Define
$$
\phi^*(z;\tau):= e^{\frac{\pi m z^2}{v}}\phi(z,\tau).
$$
Then $\phi^*$ satisfies the modular forms transformation law of an element of $J_{k,0}$.
Write
	$$
	\phi^*(z;\tau) =: \sum_{n}\psi_n^*(\overline{z};\tau) z^n.
	$$
	Using the transformation law of $\phi^\ast$,
	\begin{align*}
	\sum_n \psi_n^*\left(\frac{\overline{z}}{\overline{\tau}}; -\frac{1}{\tau} \right)\left(\frac{z}{\tau}\right)^n = \phi^*\left(\frac{z}{\tau};-\frac{1}{\tau}\right)= \tau^k \phi^*(z;\tau) = \tau^k\sum_n \psi^*_n(\overline{z};\tau)z^n.
	\end{align*}
	Thus, comparing coefficients,
	$$
	\psi_n^*\left(\frac{\overline{z}}{\overline{\tau}}; -\frac{1}{\tau} \right) = \tau^{k+n} \psi_n^* (\overline{z};\tau).
	$$
	Setting $\overline{z}=0$ in particular yields
	\begin{equation*}
	\psi^*_n\left(0;-\frac{1}{\tau}\right) = \tau^{k+n} \psi_n^*(0;\tau).
	\end{equation*}
	
	We are left to show that
	\begin{equation}\label{toshow}
	\psi_n^\ast (0,\tau) = \psi_n(\tau).
	\end{equation}
	 For this, we expand the exponential to obtain
	$$
	\phi^*(z;\tau) = \sum_n z^n \sum_{j \geq 0} \frac{\left(\frac{\pi m }{v}\right)^j}{j!}\chi_{n-2j}(\overline{z};\tau).
	$$
	Thus
	$$
	\psi_n^*(\overline{z};\tau) = \sum_{j \geq 0} \frac{\left(\frac{\pi m }{v}\right)^j}{j!}\chi_{n-2j}(\overline{z};\tau).
	$$
	Setting $\overline{z}=0$ gives \eqref{toshow}.
\end{proof}
We next turn to modular completion which use $E_2$. The proof follows as that of Proposition \ref{nonholP} since
$$
e^{\frac{\pi m z^2}{v}} e^{-\frac{\pi^2 m E_2(\tau)}{3}z^2}
$$
satisfies the modular transformation law of a Jacobi form of weight and index $0$ (which follow by using \eqref{E2trans}).
\begin{proposition}\label{3.2}
	We have
	$$
	\rho_n(\tau) := \sum_{j \geq 0} \frac{\left(\frac{\pi^2 m }{3}E_2(\tau)\right)^j}{j!}\chi_{n-2j}(0;\tau).
	$$
	transforms modular of weight $k+n$.
\end{proposition}

\section{Rank moments}\label{4}
In this section, we describe how to build modular objects out of rank moments.

\subsection{Proof of Theorem \ref{rankmain}}\label{4.1}
Using the final expression in \eqref{rank}, it follows that
$$
R(\zeta;q)=\frac{(1-\zeta)\zeta^{-\frac32}}{(q;q)_\infty}A_3(z,-\tau;\tau).
$$
Setting
\begin{multline*}
\widehat{R}(z; \tau):=\\
\left(\frac{R(\zeta; q)}{\zeta^{\frac12}-\zeta^{-\frac12}}q^{-\frac{1}{24}}+\frac12 \zeta^{-1}q^{-\frac16}S(3z+\tau; 3\tau)-\frac12 \zeta q^{-\frac16}S(3 z-\tau; 3\tau)\right)e^{-\frac{\pi^2}{2}E_2(\tau)z^2},
\end{multline*}
it is not hard to see that
\[
\widehat{R}(z; \tau)=\sum_{\ell\geq 0}r_{2\ell-1}(\tau)z^{2\ell-1}.
\]
Indeed, the contribution of $r_{2\ell-1}^-$ follows directly by definiton and for $r_{2\ell-1}^+$ we expand $R$ with \eqref{TaylorR} and use
\begin{align*}
\frac1{\zeta^\frac12-\zeta^{-\frac12}}=\sum_{n\geq -1} B_{n+1}\left(\frac12\right) \frac{(2\pi iz)^n}{(n+1)!},\quad e^{-\frac{\pi^2}{2} E_2(\tau)z^2} = \sum_{m\geq 0} \left(\frac{E_2(\tau)}{8}\right)^m \frac{(2\pi iz)^{2m}}{m!}.
\end{align*}
We then write, using Theorem \ref{2.4},
\[
\widehat{R}(z; \tau)=-\frac{\widehat{A}_3(z,-\tau;\tau)\zeta ^{-1}}{\eta (\tau)} e^{-\frac{\pi^2}{2}E_2(\tau)z^2} =-\frac{\widehat{A}_3(z, 0; \tau)}{\eta(\tau)}e^{-\frac{\pi^2}{2}E_2(\tau)z^2}.
\]
From this one can see, by  \eqref{Appellt}, \eqref{E2trans}, and \eqref{multeta}, that
\[
\widehat{R}\left(\frac{z}{c\tau+d}; \frac{a\tau+b}{c\tau+d}\right)
=\psi^{-1}\left(\begin{matrix} a&b\\c&d\end{matrix}\right) (c\tau+d)^{\frac12}\widehat{R}(z; \tau).
\]
This then directly implies \eqref{ST}.

We next show \eqref{Lsl}. Since $r_{2\ell-1}^+$ is holomorphic, $L(r_{2\ell-1}^+)=0$.
To determine the action of $L$ on $r_{2\ell-1}^-$, we first compute
\begin{equation}
\label{expandS}
\begin{aligned}
&q^{-\frac16} \zeta^{-1} S(3z+\tau;3\tau)\\
&\quad=\sum_{n\in -\frac16+\Z} \left(\sgn\left(n-\frac13\right)-E\left(\left(n+\frac{y}{v}\right)\sqrt{6v}\right)\right)(-1)^{n-\frac56}q^{-\frac{3n^2}{2}} \zeta^{-3n}.
\end{aligned}
\end{equation}
Thus
\begin{multline*}
L\left(r_{2\ell-1}^-(\tau)\right)=\frac{-1}{(2\ell-1)!}\vast[\frac{\partial^{2\ell-1}}{\partial z^{2\ell-1}} \vast(e^{-\frac{\pi^2}{2}E_2(\tau)z^2}\\
\times\sum_{n\in-\frac16+\Z}(-1)^{n-\frac56}q^{-\frac{3n^2}{2}}\zeta^{-3n}L\left(E\left(\left(n+\frac{y}{v}\right)\sqrt{6v}\right)\right)\vast)\vast]_{z=0}.
\end{multline*}
Using $E'(x)=2e^{-\pi x^2}$ gives that
\begin{align*}
L\left(E\left(\left(n+\frac{y}{v}\right)\sqrt{6v}\right)\right)= \sqrt{6}v^{\frac32}\left(n-\frac{y}{v}\right)e^{-6\pi v\left(n+\frac{y}{v}\right)^2}.
\end{align*}
Simplifying the exponential factor
and turning $e^{-\frac{6\pi y^2}{v}}$ into $e^{\frac{3\pi z^2}{2v}}$ since in the end we set $z=0$ (and thus also $\overline{z}=0$),
we obtain
\begin{align*}
&L\left(r_{2\ell-1}(\tau)\right)\\
&=\frac{-\sqrt{6}v^\frac32}{(2\ell-1)!}\left[\frac{\partial^{2\ell-1}}{\partial z^{2\ell-1}}\left(e^{-\frac{\pi^2}{2}\widehat{E}(\tau)z^2}\sum_{n\in -\frac16+\Z}(-1)^{n-\frac56}\left(n-\frac{y}{v}\right) e^{-3\pi in^2\overline{\tau}-6\pi in\overline{z}}\right)\right]_{z=0}.\\
\end{align*}
We compute
\begin{align*}
 \left[\frac{\partial^{2\ell-1}}{\partial z^{2\ell-1}}\left(e^{-\frac{\pi^2}2 \widehat E_2(\tau) z^2} \left(n-\frac{y}{v}\right)\right)  \right]_{z=0}& = -\frac{2\ell-1}{v} \left[ \frac{\partial^{2\ell-1}}{\partial z^{2\ell-1}} e^{-\frac{\pi^2}{2} \widehat E_2(\tau)z^2}\right]_{z=0} \left[ \frac{\partial}{\partial z} y \right]_{z=0}\\ &= \frac{i\left(-\frac{\pi^2 \widehat E_2(\tau)}{2}\right)^{\ell-1}}{2(\ell-1)!} (2\ell-2)! .
\end{align*}
The claim \eqref{eq:partgen} then easily follows, using that
\[
\sum_{n \in -\frac{1}{6}+\mathbb{Z}} (-1)^{n-\frac{5}{6}}e^{-3\pi in^2\overline{\tau}}=-\eta(-\overline{\tau}).
\]

\subsection{Proof of Corollary \ref{DukeCor}}
The statement follows once we show that
\begin{equation}\label{match}
\frac{1}{2\pi i} r_1(\tau)=\mathcal{M}\left(\frac{\tau}{24}\right).
\end{equation}
Firstly by definition
\begin{align*}
\frac{1}{2\pi i}r^+_1(\tau) =\mathcal{M}^+\left(\frac{\tau}{24}\right).
\end{align*}
To match the non-holomorphic parts of both sides of \eqref{match}, we need to prove that
\begin{equation}\label{nonhol3}
\frac{1}{2\pi i}\left[\frac{\partial}{\partial z}\left(\zeta^{-1}q^{-\frac16}S(3z+\tau;3\tau)\right)\right]_{z=0}=\frac{i}{4\sqrt{2}\pi}\int_{-\frac{\overline{\tau}}{24}}^{i\infty} \frac{\eta(24w)}{\left(-i\left(\frac{\tau}{24}+w\right)\right)^{\frac32}}dw.
\end{equation}
For this we show that the Fourier expansions of both sides agree.
To rewrite the left-hand side, we apply $[\frac{\partial}{\partial z}]_{z=0}$ to \eqref{expandS}, giving
\begin{equation}\label{diffz}
\sum_{n\in-\frac16+\Z}(-1)^{n-\frac56}q^{-\frac{3n^2}{2}}\left(i\sqrt{\frac{6}{v}}e^{-6\pi n^2v}
-6\pi in\left(\sgn\left(n-\frac13\right)-E\left(\sqrt{6v}n\right)\right)\right).
\end{equation}
Since $n\in-\frac16+\Z$, we have $\sgn(n-\frac13)=\sgn(n)$. We then use
\begin{equation}\label{incomplete}
\begin{aligned}
\sgn(n)-E\left(\sqrt{6v}n\right)&=\frac{\sgn(n)}{\sqrt{\pi}}\Gamma\left(\frac12, 6\pi n^2 v\right),\\
\Gamma\left(\frac12, u\right)&=-\frac12\Gamma\left(-\frac12, u\right)+\frac1{\sqrt{u}}e^{-u}.
\end{aligned}
\end{equation}
Thus the left-hand side of \eqref{nonhol3} becomes
\begin{equation}\label{incomplete exp}
\frac3{2\sqrt{\pi}}\sum_{n\in-\frac16+\Z}(-1)^{n-\frac56}|n|\Gamma\left(-\frac12, 6\pi n^2 v\right)q^{-\frac{3n^2}{2}}.
\end{equation}

Next we rewrite $\mathcal N$ in terms of the incomplete gamma function. We have
\begin{align*}
\mathcal \int_{-\overline{\tau}}^{i \infty}\frac{e^{\frac{2\pi i}{24}(6k+1)^2w}}{(-i(\tau+w))^\frac32}dw=\frac{i}2\sqrt{\frac{\pi}{3}}|6k+1|\Gamma\left(-\frac12,\frac{\pi}{6}(6k+1)^2v\right)q^{-\frac{(6k+1)^2}{24}}.
\end{align*}
From this one can then conclude that the right-hand side of \eqref{nonhol3} also equals \eqref{incomplete exp}.

\section{Joyce invariants}\label{sec:joyce}
\subsection{Proof of Theorem \ref{joycecomple}}
Setting for odd $\ell$, $g_\ell(\tau):=2 \mathbb{J}_{\ell+1} (\tau)$, it is not hard to see that, with $A_\ell$ defined in \eqref{defineAl},
$$
g_{\ell}(\tau)=\frac{1}{(2\pi i)^\ell}\lim_{w\rightarrow0}\left[\frac{\partial^\ell}{\partial z^\ell}A_2(w,z;\tau)\right]_{z=-\tau}.
$$
Moreover define
\[
\widehat{g}_\ell(\tau):= \frac{1}{(2\pi i)^\ell} \lim_{w\to0} \left[\frac{\partial^\ell}{\partial z^\ell} \left(e^{\frac{\pi zw}{v}}\widehat{A}_2(w,z;\tau)\right)\right]_{z=0}.
\]
Using \eqref{Amod} and the following identity (for $M = \left(\begin{smallmatrix} a & b \\ c & d
\end{smallmatrix} \right) \in\SL_2(\Z)$)
\begin{equation*}
\frac1{\mathrm{Im}(M\tau)}=\frac{(c\tau+d)^2}{v}-2ic(c\tau+d),
\end{equation*}
 we obtain, for $\left(\begin{smallmatrix} a & b \\ c & d
 \end{smallmatrix} \right)\in\SL_2(\Z)$,
\begin{align*}
\widehat{g}_\ell\left(\frac{a\tau+b}{c\tau+d}\right)=(c\tau+d)^{\ell+1}\widehat{g}_\ell(\tau).
\end{align*}
We next show that $\widehat{g}_\ell=2\widehat{\mathbb{J}}_{\ell+1}$. By \eqref{Appellt}, it first follows that
\begin{align*}
\widehat{g}_\ell(\tau) &= \frac{1}{(2\pi i)^\ell}\lim_{w\to 0}\left[\frac{\partial^\ell}{\partial z^\ell}\left(e^{\frac{\pi zw}{v}-2\pi iw}\widehat{A}_2(w,z-\tau;\tau)\right)\right]_{z=0}\\
&=\frac{1}{(2\pi i)^\ell}\lim_{w\to 0}\left[\frac{\partial^\ell}{\partial z^\ell}\left(e^{\frac{\pi zw}{v}}\widehat{A}_2(w,z-\tau;\tau)\right)\right]_{z=0}.
\end{align*}
We then compute by the product rule, using that $S$ is even,
\begin{equation}\label{difference}
\small
\begin{aligned}
&\widehat{g}_\ell(\tau) - g_\ell(\tau)=  \frac{1}{(2\pi i)^\ell}\sum_{j=1}^\ell \binom{\ell}{j} \left(\frac{\pi}{v}\right)^j \lim_{w\to0} w^j \left[\frac{\partial^{\ell-j}}{\partial z^{\ell-j}}A_2(w,z-\tau;\tau)\right]_{z=0}\\
&\quad+\frac{i}{2} \frac{1}{(2\pi i)^\ell}\left[\frac{\partial^\ell}{\partial z^\ell}\sum_{\nu\in\{-1,0\}}\vartheta\left(z+\nu\tau+\frac12; 2\tau\right)S\left(z+\nu\tau+\frac12; 2\tau\right)\right]_{z=0}\\
&= \frac{\delta_{\ell=1}}{4\pi v}+\frac{i}{2} \frac{1}{(2\pi i)^\ell}\vast[\frac{\partial^\ell}{\partial z^\ell}\sum_{\nu\in\{-1,0\}}\vartheta\left(z+\nu\tau+\frac12; 2\tau\right)\\
&\hspace{7cm}\quad\times S\left(z+\nu\tau+\frac12; 2\tau\right)\vast]_{z=0}
\end{aligned}
\normalsize
\end{equation}
since for $j>1$
\begin{align*}
\lim_{w\to0} w^j \left[\frac{\partial^{\ell-j}}{\partial z^{\ell-j}}A_2(w,z-\tau;\tau)\right]_{z=0}=0
\end{align*}
and
\begin{align*}
\lim_{w\to0} w\left[\frac{\partial^{\ell-1}}{\partial z^{\ell-1}}A_2(w,z-\tau;\tau)\right]_{z=0}= \begin{cases} 0 &\text{if }\ell >1,\\ -\frac{1}{2\pi i} &\text{if } \ell=1.\end{cases}
\end{align*}
Noting that $z\mapsto\vartheta(z+\frac{\nu \tau}{2}+\frac12)e^{\pi i\nu z}$ is even, (which may be seen using \eqref{tt1} and the fact that $\vartheta(z)$ is odd) \eqref{difference} becomes
\begin{multline*}
\frac{\delta_{\ell=1}}{4\pi v}+\frac{i}{2} \ \frac{1}{(2\pi i)^\ell}\sum_{\nu\in\{-1,0\}}\sum_{j=1}^{\frac{\ell-1}{2}} \binom{\ell}{2j}
\left[\frac{\partial^{2j}}{\partial z^{2j}}\vartheta_\nu(z;\tau)\right]_{z=0}\left[\frac{\partial^{\ell-2j}}{\partial z^{\ell-2j}}S_\nu(z;\tau)\right]_{z=0},
\end{multline*}
where
\begin{align*}
\vartheta_\nu(z;\tau)&:=e^{\pi i\nu z}q^\frac{\nu^2}4\vartheta\left(z+\nu\tau+\frac12;2\tau\right),\\
S_\nu(z;\tau)&:= e^{-\pi i\nu z}q^{-\frac{\nu^2}4}S\left(z+\nu\tau+\frac12;2\tau\right).
\end{align*}
We now turn the $z$-derivatives into $\tau$-derivatives. Firstly, directly from the definition, we obtain that
$$
\left(\frac{1}{2\pi i}\frac{\partial}{\partial z}\right)^2\vartheta_\nu(z;\tau)=D_\tau\left( \vartheta_\nu(z;\tau)\right),
$$
where $D_\tau:=\frac{1}{2\pi i} \frac{\partial}{\partial \tau}$.
Moreover, using that
$$
\left(4\pi i \frac{\partial}{\partial \tau}+\frac{\partial^2}{\partial z^2}\right)\left(e^{2\pi i\alpha z-\pi i \alpha^2\tau}S(z-\alpha\tau-\beta;\tau)\right)=0,
$$
yields
$$
\left[\frac{\partial^{\ell-2j}}{\partial z^{\ell-2j}}S_\nu(z;\tau)\right]_{z=0}=(-1)^{\frac{\ell-1}{2}-j}(2\pi i)^{\ell-1-2j}D^{\frac{\ell-1}{2}-j}(S_\nu(\tau)),
$$
where
\[
S_\nu(\tau):=\left[\frac{\partial}{\partial z}S_\nu(z;\tau)\right]_{z=0}.
\]
We next show that
\begin{equation}\label{showr}
S_\nu=s_\nu
\end{equation}
with $s_\nu$ given in \eqref{frn}. Rewriting
$$
S_\nu(z;\tau)=-i\sum_{n\in\frac12+\Z}\left(\sgn(n)-E\left(\left(n+\frac{\nu }{2}+\frac{y}{2v}\right)2\sqrt{v}\right)\right)q^{-\left(n+\frac{\nu}{2}\right)^2} e^{-2\pi i\left(n+\frac{\nu}{2}\right)z},
$$
we compute
\begin{align*}
S_\nu(\tau)&=-i\sum_{n\in\frac12+\Z}q^{-\left(n+\frac{\nu}{2}\right)^2}\\
&\quad\times\left(\frac{i}{\sqrt{v}}e^{-4\pi\left(n+\frac{\nu}{2}\right)^2}-2\pi i\left(n+\frac{\nu}{2}\right)\left(\sgn(n)-E\left(\left(n+\frac{\nu}{2}\right)2\sqrt{v}\right)\right)\right).
\end{align*}
Equation \eqref{showr} now follows, by using \eqref{incomplete}.
Thus \eqref{difference} becomes
\begin{align}\label{simplifyDiff}
\frac{\delta_{\ell=1}}{4\pi v}+\frac{(-1)^{\frac{\ell-1}{2}}}{4\pi}\sum_{\nu\in\{-1,0\}}\sum_{j=1}^{\frac{\ell-1}{2}} \binom{\ell}{2j}(-1)^jD^{j}\left(\vartheta_\nu(\tau)\right)
D^{\frac{\ell-1}{2}-j}\left(s_\nu(\tau)\right).
\end{align}
We next aim to prove that \eqref{simplifyDiff} equals ($f_\nu$ and $r_\nu$ have weights $\frac12$ and $\frac32$, respectively)
\begin{multline*}
\frac{\delta_{\ell=1}}{4\pi v} +\frac{(\ell-1)!(-1)^{\frac{\ell-1}{2}}}{\Gamma\left(\frac{\ell}{2}\right)^2 2^{1+\ell}}\sum_{\nu \in \{-1,0\}}[\vartheta_\nu,s_\nu]_{\frac{\ell-1}{2}}\\ =\frac{\delta_{\ell=1}}{4\pi v}+\frac{(\ell-1)!(-1)^{\frac{\ell-1}{2}}}{\Gamma\left(\frac{\ell}{2}\right)^2 2^{1+\ell}}\sum_{\nu\in \{-1,0\}} \sum_{j=0}^{\frac{\ell-1}{2}}(-1)^j \Bigg(\begin{matrix}\frac{\ell}{2}-1\\\frac{\ell-1}{2}-j\end{matrix}\Bigg)\Bigg(\begin{matrix}\frac{\ell}2\\j\end{matrix}\Bigg) D^j(\vartheta_\nu) D^{\frac{\ell-1}2-j}(s_\nu).
\end{multline*}
This follows once we show that
\begin{equation}\label{comparebin}
\binom{\ell}{2j}\frac1{4\pi} (-1)^{\frac{\ell-1}2+j}=\frac{(\ell-1)!(-1)^{\frac{\ell-1}2}}{\Gamma\left(\frac{\ell}{2}\right)^2 2^{1+\ell}}(-1)^j \Bigg(\begin{matrix}\frac{\ell}2-1\\\frac{\ell-j}{2}\end{matrix}\Bigg)\Bigg(\begin{matrix}\frac{\ell}2\\j\end{matrix}\Bigg).
\end{equation}
Now it is not hard to obtain \eqref{comparebin} by using
\begin{align*}
\Gamma(x) \Gamma\left(x+\frac12\right)=2^{1-2x}\sqrt{\pi}\Gamma(2x),\quad
\Gamma(x+1)=x \Gamma(x).
\end{align*}
Combining the above, we obtain (1.10).

We next prove (1.14). Using \eqref{difference}, we compute
$$
L\left(\widehat{g}_\ell(\tau)\right)=-\frac{\delta_{\ell=1}}{4\pi}+\frac{i}{2}\frac{1}{(2\pi i)^\ell}
\left[\frac{\partial^\ell}{\partial z^\ell}\sum_{\nu\in\{-1, 0\}}\vartheta_{\nu}(z;\tau)L\left(S_\nu(z;\tau)\right)\right]_{z=0}.
$$
Now
\[
L\left(S_\nu(z;\tau)\right)= 2iv^{\frac32} e^{-\frac{\pi y^2}{v}}\sum_{n\in\frac{\nu+1}{2}+\Z}\left(n-\frac{y}{2v}\right)e^{-2\pi in^2\overline{\tau}-2\pi in\overline{z}}.
\]
Writing $e^{-\frac{\pi y^2}{v}}=e^{\frac{\pi}{4v}\left(z^2-2z\overline{z}+\overline{z}^2\right)}$
and using that the $\vartheta_{\ell,\nu}$ are even as a function of $z$, gives \eqref{lower}, with $k=\ell+1$.

We are left to show that the $\vartheta_{\ell,\nu}(\tau)$ are almost holomorphic modular forms. Using Proposition \ref{THETtrans} gives, by a straightforward but lengthy calculation, for $\left(\begin{smallmatrix}a&b\\c&d\end{smallmatrix}\right)\in\Gamma_1(4)$,
\[
\vartheta^\ast_\nu\left(\frac{z}{c\tau+d};\frac{a\tau+b}{c\tau+d}\right)
=\chi\left(\begin{matrix}a&b\\c&d\end{matrix}\right)(c\tau+d)^{\frac12}\vartheta^\ast_\nu(z; \tau)
\]
with
\[
\vartheta_\nu^\ast(z;\tau):=\vartheta_\nu(z;r)e^{\frac{\pi z^2}{4v}}, \quad \chi\left(\begin{matrix}a&b\\c&d\end{matrix}\right):=\psi^3\left(\begin{matrix}a&2b\\ \frac{c}{2}&d\end{matrix}\right)i^{\frac c4}.
\]
We are left to show that $\vartheta_\nu^\ast$ is almost holomorphic. We have, using the product rule,
\begin{align*}
\left[\frac{\partial^{\ell-1}}{\partial z^{\ell-1}} {\vartheta}_\nu^\ast(z; \tau)\right]_{z=0}
&=\sum_{j=0}^{\ell-1}\binom{\ell-1}{j}
\left[\frac{\partial^{\ell-1-j}}{\partial z^{\ell-1-j}} \vartheta_\nu(z; \tau)\right]_{z=0}\left[\frac{\partial^j}{\partial z^j} e^{\frac{\pi z^2}{4v}}\right]_{z=0}\\
&=\sum_{j=0}^{\frac{\ell-1}2}\binom{\ell-1}{2j}\left[\frac{\partial^{\ell-1-2j}}{\partial z^{\ell-1-2j}} \vartheta_\nu(z; \tau)\right]_{z=0}\frac{\left(\frac{\pi}{4v}\right)^j}{j!}.
\end{align*}
This yields the claim since $\vartheta_\nu$ is holomorphic.

Corollary \ref{Mellit} directly follows by plugging in and using \eqref{rewritetheta}.

\end{document}